\documentclass[12pt]{amsart}
\DeclareFontFamily{OT1}{pzc}{}
\DeclareFontShape{OT1}{pzc}{m}{it}{<-> s * [1.10] pzcmi7t}{}
\DeclareMathAlphabet{\mathpzc}{OT1}{pzc}{m}{it}

\usepackage{amsfonts}
\usepackage{amssymb}
\usepackage{amsthm}
\usepackage{amsmath}
\usepackage{tikz}
\usetikzlibrary{matrix,arrows}
\usepackage{relsize}
\usepackage[backend=biber,style=numeric]{biblatex}
\defbibentryset{set}{C&V1,C&V2}

\newcommand{\sCl}{\mathrm{sCl}}

\newcommand{\gr}{\mathrm{gr}}

\newcommand{\Z}{\mathbb{Z}}
\newcommand{\N}{\mathbb{N}}
 \newsymbol\kk 207C  
 \newcommand{\la}{\langle}
\newcommand{\ra}{\rangle}

\newtheorem{theorem}[subsection]{Theorem}

\newtheorem{lemma}[subsection]{Lemma}
\newtheorem{proposition}[subsection]{Proposition}
\newtheorem{corollary}[subsection]{Corollary}
\theoremstyle{definition} 
\newtheorem{definition}[subsection]{Definition}
\newtheorem{example}[subsection]{Example}

\newtheorem{remark}[subsection]{Remark}
\theoremstyle{empty}
\newtheorem{Condition}{Condition}

 \newcommand\restr[2]{{
   \left.\kern-\nulldelimiterspace 
   #1 
   \vphantom{\big|} 
   \right|_{#2} 
   }}

\begin{document}

\title{Skew Clifford Algebras}

\author[Cassidy]{Thomas Cassidy}
\address{\centerline{Mathematics Department}
\centerline{Bucknell University, Lewisburg, PA}
\centerline{\sf tcassidy@bucknell.edu}
\newline}

\author[Vancliff]{Michaela Vancliff}
\address{\centerline{Department of Mathematics, P.O.~Box 19408}
\centerline{University of Texas at Arlington,
Arlington, TX 76019-0408}
\centerline{\sf vancliff@uta.edu}
\centerline{\sf www.uta.edu/math/vancliff}
}
\thanks{%
\baselineskip14pt%
T.~Cassidy was supported in part by grant 210199 from the Simons Foundation. 
M.~Vancliff was supported in part by NSF grants DMS-0900239 and DMS-1302050. 
M.~Vancliff is grateful to MSRI for general membership in 2013, when this project was initiated.}

\setcounter{page}{1}

\begin{abstract}
\baselineskip15pt
We introduce a generalization, called a {\em skew Clifford algebra}, of a Clifford 
algebra, and relate these new algebras to the notion of graded skew Clifford algebra 
that was defined in~2010. In particular, we examine homogenizations of skew 
Clifford algebras, and determine which skew Clifford algebras can be homogenized to 
create Artin-Schelter regular  algebras.       Just as (classical) Clifford algebras 
are the   Poincar\' e-Birkhoff-Witt (PBW) deformations of exterior 
algebras, skew Clifford algebras are the $\Z_2$-graded PBW deformations of quantum 
exterior algebras. We also determine the possible dimensions of skew Clifford 
algebras and provide several examples.  
\end{abstract}

\baselineskip18pt

\vskip-1in \maketitle

 \vspace{0.2in}

\baselineskip18pt

\section{Introduction}

An exterior algebra can be quantized, or deformed, to create a (classical) 
Clifford algebra.  
Indeed, Clifford algebras are precisely the set of    
Poincar\' e-Birkhoff-Witt (PBW) deformations of exterior algebras~(\cite{Roy}).  
An $\N$-graded analogue of a classical Clifford algebra is related to the
geometric notion of complete intersection and is called a graded Clifford algebra
in~\cite{lebruyn}; under certain geometric conditions, these algebras are
quadratic Artin-Schelter regular (AS-regular) algebras.
Although the classification 
of the AS-regular algebras of global dimension three 
is complete, the classification of the higher-dimensional cases is an open
question that continues to generate considerable interest (see, for example, 
\cite{FV,LPWZ,RZ,ZL}).
Recently, in~\cite{set}, it was shown that a skew (i.e., quantized) version of 
a graded Clifford algebra can be used to construct several new families of
quadratic AS-regular algebras of global dimension four. Since most quadratic
AS-regular algebras of global dimension three can be understood in the framework
of graded skew Clifford algebras   (cf.\ \cite{NVZ}),   it is reasonable to hope that 
the skew-Clifford approach will be useful in understanding the higher-dimensional 
cases.    

Our work herein is motivated by a desire to construct a skew version of the path
from exterior algebras to graded Clifford algebras; to this end, we define a skew
(i.e., quantized)
version of a Clifford algebra that parallels the classical definition.  We show
that these algebras are precisely the  $\Z_2$-graded  PBW deformations of quantum
exterior algebras, thereby mirroring the connection between exterior algebras and
classical Clifford algebras.  We are also interested in knowing whether the
relationship between classical Clifford algebras and graded Clifford algebras
persists in the skew case.  Can a skew Clifford algebra be homogenized to produce
an $\N$-graded quadratic algebra?  Is there a homogenization that produces a
graded skew Clifford algebra?  Does every graded skew Clifford algebra arise from
such a homogenization?  Which skew Clifford algebras can be homogenized to produce
AS-regular algebras?  Answering these questions could conceivably help with the
general project of classifying AS-regular algebras.  Moreover, given the valuable
role that Clifford algebras play in physics and geometry, it is likely that 
skew Clifford algebras will eventually find applications in physics and 
noncommutative geometry.
 
This paper is organized as follows. In Section~\ref{definitions}, we present the basic terminology and establish two equivalent ways of defining a skew Clifford algebra, parallel to the definitions of  classical Clifford algebras.   
Section~\ref{dimension} investigates the possible vector-space dimensions of skew 
Clifford algebras, and includes a variety of illustrative examples.  
Section~\ref{gradings} presents skew Clifford algebras as 
PBW deformations of the quantum exterior algebras. In Section~\ref{homogenization}, we explore possible homogenizations of skew Clifford algebras, establish connections with the graded skew Clifford algebras introduced in~\cite{set}, and determine which skew Clifford algebras can be homogenized to create AS-regular algebras.    Section~\ref{generalizations} compares our skew Clifford algebras with other generalizations of Clifford algebras found in the literature.

\bigskip
\bigskip

%
\section{Definitions}\label{definitions}
In this section, we introduce the notion of a skew Clifford algebra in 
Definition~\ref{scliff2}, and determine a universal property for the algebra in 
Theorem~\ref{scliff1}. Our approach follows the classical setting, which we first
recall.

Throughout the paper, $\kk$ denotes a field such that char$(\kk) \neq 2$,
and we denote the set of nonzero elements in $\kk$ by $\kk^\times$. 
We first recall the definition of a finitely generated classical Clifford algebra
via a universal property.
 
\begin{definition}\label{cliff1} \cite{Gar}  \ Let $V$ be a finite-dimensional vector space over $\kk$ together with a symmetric bilinear form, $\phi: V \times V \to \kk$.  A Clifford algebra associated with~$V$ and~$\phi$ is an associative $\kk$-algebra, 
$\mathrm{Cl}(V,\,\phi)$, together with a linear map,
$g:V \to \mathrm{Cl}(V,\,\phi)$, satisfying the condition 
$g(v)^2 = \phi(v,\, v)$,  for all $v\in V$, such that, for
every $\kk$-algebra~$A$, and for every linear map $f : V\to A$ with\\[-9mm]
\[
f(v)^2 = \phi(v,\,v)\cdot 1_A , 
\]
for all  $v\in V$,
there is a unique algebra homomorphism, $\bar f : \mathrm{Cl}(V,\,\phi) \to A$, such that
$f = \bar f\circ g$ as in the diagram:
\center{
\begin{tikzpicture}
\matrix (m) [matrix of math nodes, row sep=1em,
column sep=2em, text height=1.5ex, text depth=0.25ex]
{V& \mathrm{Cl}(V,\,\phi)\\ 
& \\
& A \\
};
\path[->] (m-1-1) edge node [above]{$g$} (m-1-2) ;
\path[->] (m-1-2) edge [dotted, very thick] node  [right] {$\bar f$} (m-3-2);
\path[->] (m-1-1) edge  node [below] {$f$} (m-3-2);
\end{tikzpicture}}
\end{definition} 

Since the universal property implies that $\mathrm{Cl}(V,\,\phi)$ is unique (if it exists),
one may, alternatively, define the classical Clifford algebra associated with~$V$ 
and~$\phi$   via generators and relations as follows.
\begin{definition}\label{cliff2}
Let $V$ be a  vector space with basis $\{x_1,\ldots , x_n\}$,
and let $T(V)$ be the tensor algebra of~$V$.  The Clifford algebra associated 
with~$V$ and~$\phi$ is the quotient of~$T(V)$ by the ideal generated by all elements of
the form $x_i\otimes x_j+x_j\otimes x_i-2\phi(x_i,\, x_j)\cdot 1$  for all $i, j$.
\end{definition} 

In the literature, $\mathrm{Cl}(V, \, \phi)$ is sometimes called a universal Clifford algebra.

In order to define the parallel notion of a quantized Clifford algebra, we need
some additional terminology. We use $M(n,\, \kk)$ to denote the space of all 
$n\times n$ matrices with entries in~$\kk$.

\begin{definition} \hfill\\[-7mm]
\begin{enumerate}
\item[(a)]
Following \cite[Section I.2.1]{Brown&Goodearl},   $\mu = (\mu_{ij}) \in 
M(n,\, \kk)$ is called {\it multiplicatively antisymmetric} if 
$\mu_{ij}\mu_{ji} = \mu_{ii} = 1$ for all $i, j$. 
\item[(b)]
Let $\mu\in M(n,\, \kk)$ be multiplicatively antisymmetric and let $V$ be
a vector space with basis $\mathcal B=\{x_1,\ldots , x_n\}$.  
We call a bilinear form $\phi:V\times V\to \kk$  {\em $\mu$-symmetric} (relative to 
$\mathcal B$) if $\phi(x_i,\, x_j)=\mu_{ij}\phi(x_j,\, x_i)$ for all $i, j$.   
\item[(c)] \cite{set}
Let $\mu\in M(n,\, \kk)$ be multiplicatively antisymmetric.
A matrix $M \in M(n,\, \kk)$ is 
called $\mu$-{\em symmetric} if $M_{ij} = \mu_{ij}M_{ji}$ for all $i, j$.
\end{enumerate}
\end{definition}

If $\mu_{ij}=1$ for all $i, j$, then $\mu$-symmetric bilinear forms are in fact symmetric.
Notice that the definition of $\mu$-symmetry for a bilinear form is always relative to a specified  basis.  For this reason we will typically present~$\phi$ via 
a $\mu$-symmetric matrix $B\in M(n,\, \kk)$  that 
defines the values of~$\phi$ on the relevant basis for $V$; i.e., 
$B_{ij}=\phi(x_i,\, x_j)$ for all $i, j \in \{ 1, \ldots , n\}$.

We would like to generalize Definitions \ref{cliff1} and  \ref{cliff2} by allowing the symmetric bilinear form~$\phi$ to be replaced with a $\mu$-symmetric bilinear form.  It is easy to state the $\mu$-symmetric version of Definition~\ref{cliff2}, but one should note that such algebras can be trivial; i.e., of dimension zero
(cf.\ Example~\ref{zerodim}).

\begin{definition}\label{scliff2}
Let $V$ be a vector space with ordered basis $\{x_1, \ldots , x_n\}$, and 
let $\phi$~be a $\mu$-symmetric bilinear form associated with this basis.   
The skew Clifford algebra $\sCl(V,\, \mu,\, \phi)$ associated with~$\phi$ 
is the quotient of the tensor algebra~$T(V)$ by the ideal generated by all elements 
of the form 
$x_i\otimes x_j+\mu_{ij}x_j\otimes x_i-2\phi(x_i, \, x_j)\cdot 1$  for all $i, j$.
\end{definition} 

If $\mu_{ij}=1$ for all $i, j$, then $\sCl(V,\, \mu,\,  \phi)$ is a classical Clifford 
algebra. 
If $\phi \equiv 0$, then $\sCl(V,\, \mu,\, \phi)$ is a quantum exterior algebra (see
\cite{Brown&Goodearl}) which we denote by $\Lambda_\mu(V)$.  We sometimes refer to
$\sCl(V,\, \mu, \, \phi)$ as the skew Clifford algebra associated to~$\mu$ 
and~$B$, where $B$~is the $\mu$-symmetric matrix determined by~$\phi$ relative to 
the basis $\{x_1,\ldots , x_n\}$.

\begin{remark}
Let $R$ be the skew Clifford algebra associated to matrices $\mu$ and $B$. From the defining relations in Definition \ref{scliff2}, we see that the opposite 
algebra~$R^{\mathrm{op}}$ 
is the skew Clifford algebra associated with the transposed matrices 
$\mu^{\mathrm{t}}$ and $B^{\mathrm{t}}$.
\end{remark}

The inclusion of $V$ into $T(V)$, composed with the natural projection of $T(V)$
onto $\sCl(V,\, \mu,\, \phi)$, creates a map which we call~$g$.  
\begin{figure}[h]
\center{ 
\begin{tikzpicture}
\matrix (m) [matrix of math nodes, row sep=1em,
column sep=2em, text height=1.5ex, text depth=0.25ex]
{V& T(V)\\ 
& \\
& \sCl(V,\, \mu,\, \phi)\\
};
\path[right hook->] (m-1-1) edge (m-1-2) ;
\path[->>] (m-1-2) edge [thick] node  [right] {$\pi$} (m-3-2);
\path[->] (m-1-1) edge  node [below] {$g$} (m-3-2);
\end{tikzpicture}}
\end{figure}
In the case of a (universal) Clifford algebra, $g$ is necessarily injective, but
this need not be true for skew Clifford algebras in general 
(cf.\ Example \ref{zerodim}).  A description of a skew Clifford algebra,  parallel to  Definition \ref{cliff1}, can be made in a straightforward manner when the map $g$ is injective.  We prove in Theorem \ref{tfae} that the map
$g:V\to \sCl(V,\, \mu,\, \phi)$ is injective  exactly when the skew Clifford algebra
  has dimension $2^{\dim(V)}$.

\begin{theorem} \label{scliff1}
Let $V$ be a vector space over~$\kk$ with fixed ordered basis 
$\{x_1, \ldots , x_n\}$ together with a $\mu$-symmetric bilinear form, $\phi: V
\times V \to \kk$.  If $g:V\to \sCl(V,\, \mu, \, \phi)$ is injective, then  
$\sCl(V,\, \mu, \, \phi)$ satisfies the condition 
\[
g(x_i)g(x_j)+\mu_{ij}g(x_j)g(x_i) = 2\phi(x_i,\, x_j)\cdot 1
\]
for all $i, j\in \{1,\ldots , n\}$, and, for every $\kk$-algebra $A$, and for every 
injective linear map $f : V\to A$, with 
\[
f(x_i)f(x_j)+\mu_{ij}f(x_j)f(x_i) = 2\phi(x_i, \, x_j)\cdot 1_A
\] 
for all $i, j\in\{1,\ldots , n\}$, there is a unique algebra 
homomorphism, $\bar f : \sCl(V,\, \mu, \, \phi) \to A$, such that
$f = \bar f\circ g$ as in the diagram$:$
\center{
\begin{tikzpicture}
\matrix (m) [matrix of math nodes, row sep=1em,
column sep=2em, text height=1.5ex, text depth=0.25ex]
{V& \sCl(V,\, \mu, \, \phi)\\ 
& \\
& A \\
};
\path[ ->] (m-1-1) edge node [above]{$g$} (m-1-2) ;
\path[->] (m-1-2) edge [dotted, very thick] node  [right] {$\bar f$} (m-3-2);
\path[->] (m-1-1) edge  node [below] {$f$} (m-3-2);
\end{tikzpicture}}
\end{theorem} 
\begin{proof}
The defining relations of $\sCl(V,\, \mu, \, \phi)$ guarantee that 
\[g(x_i) g(x_j) +\mu_{ij}g(x_j)g(x_i)= 2\phi(x_i, \, x_j)\cdot 1,\] 
for all $i, j\in\{1,\ldots , n\}$.  Let $A$~be a $\kk$-algebra and $f : V\to A$ 
be an injective linear map such that 
\[f(x_i)f(x_j)+\mu_{ij}f(x_j)f(x_i) = 2\phi(x_i, \, x_j)\cdot 1_A,\] 
for all $i, j$.  
Since the set $\{g(x_i)\}$ generates $\sCl(V,\, \mu, \, \phi)$, it suffices to define 
$\bar f$ on these generators.  Set $\bar f(g(x_i)):=f(x_i)$ for all $i$.  Since 
$g$~is injective, $\bar f$ is well defined.  Uniqueness  follows from the fact 
that $\{g(x_i)\}$ generates $\sCl(V,\, \mu, \, \phi)$.
\end{proof}

\bigskip 
\bigskip


\section{Dimension}\label{dimension}

In this section we explore the possible dimension of a skew Clifford algebra.  A 
universal Clifford algebra on $n$~generators has dimension $2^n$.  In contrast, 
a skew Clifford algebra on $n$~generators can have much smaller dimension.

We use Bergman's Diamond Lemma, from \cite{bergman}, to calculate the dimension of
skew Clifford algebras.   Let $\sCl(V,\, \mu,\, \phi)$ be a skew Clifford algebra
as in Definition~\ref{scliff2}, and write $B_{ij}=\phi(x_i,\, x_j)$ for all 
$i, j\in\{1,\ldots , n\}$, where $\{x_1, \ldots, x_n\}$ is a fixed ordered basis 
of~$V$.  We order the generators so that $x_i<x_j$ if $i<j$.  The defining relations 
for $\sCl(V,\, \mu,\, \phi)$ provide initial ambiguities of the
form $x_j^2 x_i$, $x_j x_i^2$ and $x_k x_j x_i$ for all $i<j<k$.  These ambiguities are all resolvable provided  the following linear equations hold for all distinct $i, j$ 
and $k$:
\begin{equation}\label{eq1}
2(1-\mu_{ij})B_{ij}x_i=(1-\mu^2_{ij})B_{ii}x_j,
\end{equation}
\begin{equation}\label{eq2}
2(1-\mu_{ij})B_{ij}x_j=(1-\mu^2_{ij})B_{jj}x_i,
\end{equation}
\begin{equation}\label{eq3}
(1-\mu_{ij}\mu_{ik})B_{jk}x_i+(\mu_{ij}-\mu_{jk})B_{ik}x_j=(1-\mu_{jk}\mu_{ik})B_{ij}x_k.
\end{equation}
(One could use equation (\ref{eq3}) with repeated indices to consolidate all three
equations.) These equations hold in $\sCl(V,\, \mu,\,\phi)$ without introducing new linear relations if and only if the coefficients in (\ref{eq1}), (\ref{eq2}) and (\ref{eq3}) are all zero.  This is equivalent to the following condition:

\noindent $\ \bigstar$ \ \ \ \ \ for all $i$ and $j$, if $B_{ij}\ne 0$ then   
$\mu_{ik}=\mu_{kj}$   for all $k$.

\begin{lemma}\label{span}
The dimension of a skew Clifford algebra $\sCl(V,\, \mu,\, \phi)$ is at
most $2^{\dim(V)}$.
\end{lemma}

\begin{proof}
Let $\{x_1,\ldots , x_n\}$ be a basis for $V$.  Since, in $\sCl(V,\, \mu, \, \phi)$,
we have 
\[
g(x_i)g(x_j)+\mu_{ij}g(x_j)g(x_i)=2\phi(x_i,\, x_j)\cdot 1,
\]
for all $i, j$, it is clear that $1$ together with the products
\begin{equation}\label{spanset}
\qquad \qquad \qquad 
g(x_{i_1})g(x_{i_2})\cdots g(x_{i_k}), 
\qquad \qquad 1\le i_1<i_2<\cdots<i_k\le n,
\end{equation}
span $\sCl(V,\,  \mu, \, \phi)$.  This gives a spanning set with $2^n$ elements, and so 
the dimension of $\sCl(V,\, \mu, \, \phi)$ is at most~$2^n$.  
\end{proof}

We say that a skew Clifford algebra $\sCl(V,\, \mu, \, \phi)$ 
has {\it full dimension} if its dimension is
$2^{\dim(V)}$, and we say that it is {\it trivial} if its dimension is~$0$.
Clifford algebras are examples of skew Clifford algebras with full dimension, but
there exist many skew Clifford algebras of full dimension that are not Clifford
algebras, as illustrated in Examples \ref{fdex1}, \ref{fdex2}, \ref{fdex3} and
\ref{notsimple}.  Examples \ref{inbetweenex1} and \ref{betweenex2}   demonstrate
that dimensions strictly between zero and $2^{\dim(V)}$ can occur.

\begin{remark}\label{diag}
As a consequence of condition $\bigstar$, if $R$~is a skew Clifford algebra of
full dimension such that $B \in M(n, \, \kk)$~is diagonal of rank~$n$, 
then $\mu$~is a symmetric matrix with $\mu_{ij}^2=1$ for all $i, j$.
\end{remark}

\begin{example}\label{fdex1}
 For any $a\in\kk^\times$ and $b\in \kk$, let 
$\mu= 
\footnotesize \begin{pmatrix}
      1&a & 1 \\ a^{-1}&1& a \\ 1&a^{-1}&  1
     \end{pmatrix}$ 
and $B= \footnotesize \begin{pmatrix} 
            0&0 & b \\ 0&0& 0  \\ b&0& 0 
	    \end{pmatrix}$.  
Up to isomorphism, we may take $b \in \{0, \, 1\}$.  These algebras have 
dimension~$8$.
\end{example}

\begin{example}\label{fdex2}
If $\mu$ is any multiplicatively antisymmetric matrix and $B = 0$,
then $\sCl(V,\, \mu, \, \phi) = \Lambda_\mu(V)$,  the quantum exterior algebra.   
Quantum exterior algebras have full dimension.
\end{example}

\begin{example}\label{fdex3}
Let $\mu = (\mu_{ij})$ be a multiplicatively antisymmetric matrix such that 
$\mu_{ij} = -1$, for all $i \neq j$, and let $B$ be any diagonal matrix. Up to 
isomorphism, we may take the diagonal entries in $B$ to be either $0$ or $1$.  
These skew Clifford algebras of full dimension are quotients of polynomial rings.
\end{example}

\begin{example}\label{inbetweenex1}
Let 
\[
\mu= \footnotesize
\begin{pmatrix}
      1&\mu_{12}& \mu_{13} & 1 \\
      \mu_{21}&1&\mu_{23} & 1 \\
      \mu_{31}&\mu_{32}& 1 & 1 \\
      1&1& 1 & 1 
\end{pmatrix}  
\normalsize
\quad \text{and} \quad 
B= \footnotesize
\begin{pmatrix}
      0&0& 0 & B_{14} \\
      0&0& 0 & B_{24} \\
      0&0& 0 & B_{34} \\
      B_{14}&B_{24}& B_{34} & 1 
\end{pmatrix},
\]
where $B_{i4}\neq 0$, for all $i$.              
If $\mu_{23}=\mu_{13}=1\ne \mu_{12}$, then this algebra has dimension~$8$.  
If either $\mu_{13}$ or $\mu_{23}$ is not equal to $1$, then the algebra has 
dimension~$4$.
\end{example}

\begin{example}\label{zerodim}
Let $\mu=\footnotesize 
\begin{pmatrix} 
        1&-1\\ -1&1 
\end{pmatrix}$ 
and 
$B=\footnotesize 
\begin{pmatrix} 
      0&1\\ -1&0 
\end{pmatrix}$.  
This algebra has dimension zero.
\end{example}

\begin{example}\label{notsimple}
Let $\mu=\footnotesize 
\begin{pmatrix}
      1&1&-1\\ 1&1&-1\\ -1&-1&1 
\end{pmatrix}$ 
and   
$B= \footnotesize 
\begin{pmatrix} 
0&a&0\\ a&0&0\\ 0&0&1 
\end{pmatrix}$,  
for some $a \in \kk$.
This algebra has dimension~$8$.  We note that the algebra is not simple, since 
$x^3+1$ generates a proper ideal. 
\end{example}

We note that, for generic data $(\mu,\, \phi)$, where $\mu$ is a multiplicatively 
antisymmetric matrix and $\phi$ is a $\mu$-symmetric bilinear form, the associated 
skew Clifford algebra has dimension zero.  

The Koszul dual, $\Lambda_\mu(V)^!$, of $\Lambda_\mu(V)$ is a skew polynomial ring
and the matrix~$B$, in this case,  defines a quadratic element 
$q \in \Lambda_\mu(V)^!$ as follows.  
Using $z_i$ to denote the generator in $\Lambda_\mu(V)^!$ dual to $x_i$, and
writing $z=(z_1, \ldots , z_n)$, we define $q:=zBz^{\mathrm{t}}$.
 
\begin{theorem}\label{tfae}
With notation as in Definition~\ref{scliff2}, let 
$\mathcal B = \{x_1,\ldots , x_n\}$ and let $B$ denote the matrix of~$\phi$ with 
respect to~$\mathcal B$. The following are equivalent$:$
\begin{enumerate}
\item[(a)] $g:V\to \sCl(V,\,  \mu, \, \phi)$ is injective,
\item[(b)] $\sCl(V,\, \mu, \, \phi)$ has dimension $2^n$,
\item[(c)] the coefficients in equations (\ref{eq1}), (\ref{eq2}) and (\ref{eq3}) 
are all zero,
\item[(d)]  in the skew polynomial ring $\Lambda_\mu(V)^!$, the element $q$ 
determined by the matrix~$B$ is central.
\end{enumerate}
\end{theorem}
\begin{proof}
$(a)\Rightarrow(c)$.  
The set $\{ g(x_1), \ldots , g(x_n) \}$ spans $g(V)$ in $\sCl(V,\, \mu, \, \phi)$.  
Any linear relation $\sum_i \alpha_i g(x_i) = 0$, 
$\alpha_i\in\kk$, would imply that $\sum_i \alpha_i x_i \in \ker(g)$.
Since $g$ is injective, there can be no such nontrivial linear relation, and hence the 
coefficients in equations (\ref{eq1}), (\ref{eq2}) and (\ref{eq3}) are all zero.

$(c)\Rightarrow(b)$.  Since the coefficients in equations  (\ref{eq1}),
(\ref{eq2}) and (\ref{eq3}) are all zero, each of the $2^n-1$ monomials in the set
(\ref{spanset}) is fully reduced, and hence the dimension of 
$\sCl(V,\,  \mu, \, \phi)$ is~$2^n$.  

$(b)\Rightarrow(a)$.  If $\sCl(V,\, \mu, \, \phi)$ has dimension~$2^n$, then the 
spanning set in the proof of Lemma~\ref{span} is a basis.  This basis contains the set
$\{g(x_1),\ldots , g(x_n)\}$, and the linear independence of these elements implies 
that $g$~is injective.

$(c)\Leftrightarrow(d)$. 
If $B=0$, then the coefficients in equations (\ref{eq1}), (\ref{eq2}) and 
(\ref{eq3}) are all zero and $q=0$, so (c) and (d) both hold.
Suppose $B \neq 0$. In particular, $B_{ij}\ne 0$ for some $i, j$, so the monomial 
$z_i z_j$ appears in the expression for~$q$.  It follows that $q$ is central if and 
only if $z_i z_j$ is central for all $i, j$ such that $B_{ij} \neq 0$. 
This happens if and only if, for each $k$, we have 
$(z_i z_j)z_k=z_k(z_i z_j)=\mu_{ik}\mu_{jk}z_i z_j z_k$ 
for all $i, j$ such that $B_{ij} \neq 0$; 
that is, for each $k$, we have $\mu_{ik}=\mu_{kj}$ 
whenever $B_{ij} \neq 0$.  
By condition~$\bigstar$, this is equivalent to all the coefficients in equations 
(\ref{eq1}), (\ref{eq2}) and (\ref{eq3}) being zero.
\end{proof}

\begin{corollary}\label{symmetric}  
With notation as in Theorem~\ref{tfae}, if $\sCl(V,\, \mu, \, \phi)$ has full 
dimension, then $B$~is a symmetric matrix.  
\end{corollary}

\begin{proof} By Theorem \ref{tfae}, the hypothesis implies $\bigstar$ holds, which implies that $\mu_{ij}=1$ whenever $B_{ij}\ne 0$.  Hence, $B_{ji}=B_{ij}$ for all $i, j$.
\end{proof}

The following example shows that the converse to Corollary \ref{symmetric} does not hold.

\begin{example}\label{betweenex2}
Suppose char$(\kk) \notin \{ 2, \, 3\}$,  and let 
$\mu=\scriptsize \begin{pmatrix} 1&2\\ 1/2&1 \end{pmatrix}$ and 
$B= \scriptsize \begin{pmatrix} 0&0\\ 0&1 \end{pmatrix}$. 
Here $B$ is symmetric, but the algebra has dimension $2$, since the relations 
$x_1 x_2 = -2 x_2 x_1$ and $x_2^2=1$ imply that $x_1=0$.
\end{example}

\begin{theorem}\label{subspace}
With notation as in Definition~\ref{scliff2} and Theorem~\ref{tfae}, if 
$\sCl(V,\, \mu, \, \phi)$ is nontrivial, then $\sCl(V,\, \mu, \, \phi)$ is isomorphic 
to $\sCl(V',\, \mu', \, \restr{\phi}{V'})$, 
where $\mu'$~is the multiplicatively antisymmetric submatrix of~$\mu$ corresponding to 
a subset~$\mathcal B'$ of~$\mathcal B$, $V'$~is spanned by~$\mathcal B'$, and 
$\restr{g}{V'} : V'\to \sCl(V,\, \mu, \, \phi)$ is injective.  
\end{theorem} 
\begin{proof}
If the coefficients in equations  (\ref{eq1}), (\ref{eq2}) and (\ref{eq3}) are all zero, then by Theorem~\ref{tfae}, $g$~is injective and we take $V'=V$ and $\mu'=\mu$.
 
Assume from now on that the coefficients in equations  (\ref{eq1}), (\ref{eq2}) and (\ref{eq3}) are not all zero.  We proceed by induction on the size of~$\mathcal B$.  
If $|\mathcal B|=1$, then $\sCl(V,\, \mu, \, \phi)=\kk [x]/ \langle x^2-b\rangle$ for 
some $b\in\kk$.  This algebra has basis $\{1, \, x\}$ and dimension~$2$, and 
clearly $g$~is injective. 
 
Now let $\mathcal B= \{x_1,\ldots , x_n\}$ where $n>1$. 
Assume inductively that a skew Clifford algebra associated to a vector space of 
dimension less than~$n$ has dimension~$2^l$ for some $l<n$ and is isomorphic to 
a skew Clifford algebra associated with the span of a subset of~$\mathcal{B}$ on 
which $g$~is injective.     
Since equations (\ref{eq1}), (\ref{eq2}) and (\ref{eq3})  include at least one 
nonzero coefficient, there are three possible consequences in 
$\sCl(V,\, \mu, \, \phi)$: 
\begin{enumerate}
\item[(i)] $x_j = 0$ for some $j$; 
\item[(ii)] $x_k\ne 0$ for all $k$, but $x_j \in \kk^\times x_i$ for some $i \neq j$;
\item[(iii)] (i) and (ii) are false and equation~(\ref{eq3}) implies 
$x_k \in \kk^\times x_i + \kk^\times x_j$ for some $i$, $j$, $k$.  
\end{enumerate}

(i) First we consider the possibility that  $x_j = 0$ in $\sCl(V,\, \mu, \, \phi)$ 
 for some~$j$.  Since $\sCl(V,\, \mu, \, \phi)$ is nontrivial, each relation of the form 
$x_j x_m+\mu_{jm}x_m x_j=2B_{jm}$ vanishes. It follows that $B_{jm}=0 = B_{mj}$ 
for all $m$. Hence, $x_j$ plays no role in $\sCl(V,\, \mu, \, \phi)$, so we may replace 
$\mathcal B$ with the subset $\mathcal B \setminus \{x_j\}$.  The induction
hypothesis implies that the algebra has dimension~$2^l$ for some $l\le n-1$ and 
is equal to $\sCl(V',\, \mu', \, \restr{\phi}{V'})$ where $\restr{g}{V'}$~is injective 
and $V'$~is spanned by a subset of $\mathcal B\setminus\{x_j\}$.  
   
(ii) Next we assume (ii) holds.  Suppose $x_j=\lambda x_i$ where 
$\lambda\in\kk^\times$. If $x_k \in \kk^\times x_i$ for all $k$, 
then there exists $b\in\kk$ such that $\sCl(V,\, \mu,\, \phi) \cong 
\kk [x]/ \langle x^2-b\rangle$, which has dimension $2$ and the result follows. 
So we assume that there exists~$k$ such that $x_k\notin\kk x_i$.    We also
assume, without loss of generality, that $j > i$. The relation 
$x_j = \lambda x_i$ introduces additional ambiguities to be resolved via 
Bergman's Diamond Lemma; namely, $x_j x_i$, $x_j^2$, $x_j x_k$, for all $k < j$, 
and $x_k x_j$ for all $k > j$. Any relations produced, by resolving these
ambiguities, that involve only scalars are resolvable, since the algebra is
nontrivial. The other relations produced by resolving these ambiguities have the
form
\begin{equation}\label{bad}
(\mu_{jk}-\mu_{ik}) \lambda x_k x_i= 2(B_{jk}-\lambda B_{ik}),
\end{equation}
for all $k < j$, where $x_k \notin \kk x_i$, 
and 
\begin{equation}\label{bad2}
(\mu_{kj}-\mu_{ki}) \lambda x_i x_k= 2(B_{kj}-\lambda B_{ki}),
\end{equation}
for all $k > j$, where $x_k \notin \kk x_i$.  
Multiplying (\ref{bad})  on the left by $x_k$  yields 
$(\mu_{jk}-\mu_{ik})\lambda B_{kk} x_i=2 (B_{jk}-\lambda B_{ik} )x_k$.  
As $x_k \notin \kk x_i$, it follows that 
\begin{equation}\label{this}
 B_{jk}=\lambda   B_{ik}
 \end{equation}
and that $B_{kk}(\mu_{jk}-\mu_{ik})=0$, 
for all $k < j$ such that $x_k \notin \kk x_i$.
Similarly, multiplying (\ref{bad}) on the right by $x_j$ (respectively, $x_i$) 
implies that 
$B_{jj}( \mu_{jk}-\mu_{ik})=0=B_{ii}(\mu_{jk}-\mu_{ik})$ 
for all $k < j$ such that $x_k \notin \kk x_i$. Moreover, similar computations
using (\ref{bad2}) yield analogous formulae
for all $k > j$ such that $x_k \notin \kk x_i$.

Hence, if $\mu_{jk} = \mu_{ik}$ for all $k$ with the property that 
$x_k \notin \kk x_i$,
then all the ambiguities that involve $x_j$ are resolvable, which means that 
the relation $x_j = \lambda x_i$  contributes no additional relations. In this
case, it then follows that $x_j$ is redundant in $\sCl(V,\, \mu, \, \phi)$, and so 
the result follows as in~(i). Thus, in order to complete the proof of case~(ii),
it suffices to prove that $\mu_{jk} = \mu_{ik}$ if $x_k \notin \kk x_i$.

Suppose, for a contradiction, that $x_k \notin \kk x_i$, $k < j$ and 
$\mu_{jk}\ne\mu_{ik}$.
It follows that $0=B_{ii}=B_{jj}=B_{kk}$.  Since $B_{ii}=0$, we obtain 
\[2B_{ij}=x_ix_j+\mu_{ij}x_jx_i=\lambda(1+\mu_{ij})x_i^2= 
\lambda(1+\mu_{ij})B_{ii}=0.\]  
Considering equation~(\ref{eq1})  for the pairs $(i,\, k)$ and $(j,\, k)$, and
noting that the right side of the equation is zero in each case, we find that 
$(1-\mu_{ik})B_{ik}=0=(1-\mu_{jk})B_{jk}$.  As $\mu_{ik}$ and $\mu_{jk}$ cannot 
both be equal to $1$, we obtain that $B_{jk}$ or $B_{ik}$ is zero, and so, 
by~(\ref{this}), both are zero.    This means that $B_{cd} = 0$ for all $c, d \in
\{i, \, j, \, k\}$. It follows that, in order to have $x_j \in \kk^\times x_i$, 
there exists $p \notin \{i, \, j, \, k\}$ such that $B_{ip}\ne 0\ne B_{jp}$.  
Considering equation~(\ref {eq1}) with the subscript pairs $(i,\, p)$ and $(j,\, p)$, 
and noting that the right side of the equation is zero in each case, we find that 
$\mu_{ip}=1=\mu_{jp}$.   Similarly, considering equation~(\ref{eq3}) for the subscript 
triple $(i, \, p, \, k)$, and noting that the middle term vanishes since 
$B_{ik}=0$, and using the assumption that $x_k \notin \kk x_i$ and the fact that 
$B_{ip}\ne 0$, we obtain $1-\mu_{pk}\mu_{ik}=0$.  However, a similar discussion 
using equation~(\ref{eq3}) with the subscript triple $(p, \, j, \, k)$ implies that
$1-\mu_{pk}\mu_{jk}=0$, which is a contradiction. Similarly, a contradiction is
obtained if one assumes $k > j$ instead.

(iii) Lastly, we assume that (i) and (ii) are false and that 
equation~(\ref{eq3}) implies $x_k \in \kk^\times x_i + \kk^\times x_j$ for some 
$i$, $j$, $k$. In particular, all three coefficients in equation~(\ref{eq3}) are 
nonzero for the subscript triple $(i, \, j, \, k)$, and 
all the coefficients in equations~(\ref{eq1}) and~(\ref{eq2})  
are zero.  Applying this to the subscript pairs $(i, \, j)$ and $(j, \, k)$, 
we obtain $B_{ij}(\mu_{ij} -1) = 0 = B_{jk}(\mu_{jk} -1)$.  It follows that 
at least one of the coefficients in equation~(\ref{eq3}) is zero, which is a
contradiction. 
\end{proof}

\begin{corollary}
With notation as in Definition~\ref{scliff2}, if 
$\sCl(V,\, \mu,\, \phi) \neq \{0\}$, then\\
$\dim(\sCl(V,\, \mu,\, \phi)) = 2^j$, where $1\le j\le \dim(V)$.
\end{corollary}
\begin{proof}
By Theorem~\ref{subspace}, the map $g$ is injective on a subspace $V'$ of $V$.  Let 
$j=\dim(V')$ and apply Theorem \ref{tfae}.
\end{proof}

Since, by Theorem~\ref{subspace},
any nontrivial skew Clifford algebra can be identified with a skew Clifford 
algebra of full dimension, we will often find it convenient to assume that the 
skew Clifford algebra under discussion has full dimension.  When we do so,
Corollary~\ref{symmetric} implies that we may also assume that  $\phi$ is a
symmetric bilinear form, and we can define the corresponding quadratic form on~$V$
by $\Phi(x):=\phi(x, \, x)$ for all $x\in V$.

\bigskip 
\bigskip 


\section{Gradings and  Deformations}\label{gradings}

Skew Clifford algebras have a $\mathbb Z_2$-grading which is obtained by assigning each  $x_i$ an odd degree, and this makes each skew Clifford algebra a superalgebra  that can be decomposed into the direct sum of its odd and even parts.  As with classical Clifford algebras, the even part forms a subalgebra for which the odd part is a module.   However, skew Clifford algebras are typically not supercommutative.

Classical Clifford algebras are quantizations  of the exterior algebra in that,
for a given vector space $V$, each classical Clifford algebra is  isomorphic as a
vector space to the exterior algebra on~$V$, although the two algebras have
different multiplications.  As a consequence of Theorem~\ref{subspace}, each
nontrivial skew Clifford algebra $\sCl(V,\, \mu, \, \phi)$  is a quantization of 
the quantum exterior algebra~$\Lambda_{\mu'}(V')$, where $V'$~is a subspace 
of~$V$ and $\mu'$~is the corresponding multiplicatively antisymmetric submatrix 
of~$\mu$.   

Any skew Clifford algebra $R=\sCl(V,\, \mu,\, \phi)$ is also a nonhomogeneous
quadratic algebra  in the sense of Braverman and Gaitsgory in~\cite{B&G}, which 
is obtained by deforming the relations of the associated quantum exterior algebra.
Specifically, consider the natural filtration $F^iT=\{\oplus T^j (V): j\le i\}$
on the tensor algebra  $T(V)$, and let $\pi: F^2T\to V\otimes V$ denote the   
projection onto $V \otimes V$. If $P\subset F^2T$ is a subspace such 
that $R={T(V)}/{\la P\ra}$, then the quantum exterior algebra~$\Lambda_\mu(V)$ can 
be represented as~${T(V)}/{\la \pi(P)\ra}$, and every $\Z_2$-graded nonhomogeneous 
quadratic algebra obtained by deforming the relations of~$\Lambda_\mu(V)$ is
a skew Clifford algebra.

Moreover, $R$ inherits a filtration from $T(V)$, and so we can consider the
associated graded algebra~$\gr(R)$. Following~\cite{B&G}, $R$ is called a 
PBW~deformation of~$\Lambda_\mu(V)$ if the natural projection 
$\Lambda_\mu(V)\to \gr(R)$ is an isomorphism.
In the case of classical Clifford algebras, the associated graded algebra is
isomorphic to the exterior algebra, making every Clifford algebra a PBW
deformation  of the corresponding exterior algebra.  In fact, classical Clifford
algebras are exactly the set of  PBW deformations of  exterior
algebras~(\cite{Roy}).   It is natural to ask whether skew Clifford algebras are PBW
deformations of quantum exterior algebras; i.e., is the associated graded algebra of a skew Clifford algebra the same as the quantum exterior algebra obtained by 
replacing~$\phi$ with~$0$?   The next result shows that every nontrivial skew 
Clifford algebra is a PBW~deformation of some quantum exterior algebra. 

\begin{proposition}\label{PBW}
With notation as in Definition~\ref{scliff2}, let $\sCl(V,\, \mu, \, \phi)$ be  
nontrivial. Any $\Z_2$-graded PBW deformation of the quantum exterior 
algebra $\Lambda_{\mu}(V)$ is a non-trivial skew Clifford algebra, and  $\sCl(V,\, \mu, \, \phi)$ is a PBW~deformation of a quantum exterior 
algebra~$\Lambda_{\mu'}(V')$, where $V'$~is a subspace of~$V$ and $\mu'$~is the 
corresponding multiplicatively antisymmetric submatrix of~$\mu$. 
\end{proposition}
\begin{proof}
By Theorem~\ref{subspace}, $\sCl(V,\, \mu, \, \phi)\cong \sCl(V',\, \mu', \, \phi')$, 
where $\mu'$~is the multiplicatively antisymmetric submatrix of~$\mu$ corresponding 
to the subspace~$V'$ of~$V$,  
$\phi'$ is the corresponding restriction of~$\phi$,  and $\sCl(V',\, \mu', \, \phi')$ 
has full dimension.  Let  $P\subset F^2T(V')$ be a subspace such that 
$\sCl(V',\, \mu', \, \phi') = T(V')/\la P \ra$.  It follows that $P$~is spanned by 
elements of the form 
$x_i\otimes x_j+\mu_{ij}x_j\otimes x_i-2\phi(x_i, \, x_j)\cdot 1$.  
The quantum exterior algebra $\Lambda_{\mu'}(V')$ can be expressed as the 
quotient $T(V')/\la \pi (P) \ra$, which is the Koszul dual of a skew polynomial 
ring, and hence is Koszul.
This fact, combined with the Jacobi-type conditions given 
in \cite[Theorem~4.2]{cassidy&shelton}, tells us  that $T(V')/\la P \ra$ is a 
PBW~deformation of  $T(V')/\la \pi (P) \ra$ if and only if 
$(V'P+PV'+P) \cap F^2T(V')$ is contained in~$P$.  Since elements of~$P$ have 
even degree,  $V'P+PV'$ has odd degree, and consequently the containment  
$(V'P+PV'+P) \cap F^2T(V')\subset P$ is equivalent to $(V'P+PV') \cap F^2T(V')=0$.   
It follows that $V'P+PV'$ contains degree-one elements  if and only if the relations 
in~$P$ allow us to reduce a cubic monomial in two distinct ways.
Thus, the condition $(V'P+PV') \cap F^2T(V')=0$ is the same as requiring the 
ambiguities of $P$ to be resolvable. However, the ambiguities are resolvable if 
and only if the coefficients in equations (\ref{eq1}), (\ref{eq2}) and (\ref{eq3}) 
are all zero, and Theorem~\ref{tfae} tells us that this is equivalent to
$\sCl(V',\, \mu', \, \phi')$ having full dimension. We conclude  that 
$\sCl(V',\, \mu', \, \phi')$ is a PBW~deformation of $\Lambda_{\mu'}(V')$. 

Any $\Z_2$-graded  deformation $T(V)/\la P \ra$ of a quantum exterior 
algebra $T(V)/\la \pi (P) \ra$  is a skew Clifford algebra  where elements of~$P$ 
have even degree. The reversibility of the  conditions above implies that 
if $T(V)/\la P \ra$  is a PBW deformation, then the ambiguities of $P$ are 
resolvable and  $T(V)/\la P \ra$ has full dimension. Thus, non-trivial 
skew Clifford algebras are exactly the set of $\Z_2$-graded PBW deformations of 
quantum exterior algebras. 
\end{proof}

\bigskip


\section{Homogenizations}\label{homogenization}

In this section we show how skew Clifford algebras can be homogenized to create $\N$-graded   algebras, and we identify the skew Clifford algebras that correspond to $\N$-graded algebras with good homological  properties.

A skew Clifford algebra $R=\sCl(V,\, \mu, \, \phi)$ of full dimension can be homogenized to create an $\N$-graded algebra $A$   by adjoining a single degree-two central 
generator~$y$.  We set
$$A:=\frac{\kk\la X_1,\ldots , X_n\ra[y]}{\la X_iX_j+\mu_{ij}X_jX_i-2B_{ij}y : 1\le 
i, j\le n\ra},$$
  where the square brackets indicate that $y$~is central in this algebra.
Clearly $A/(y-1)A\cong R$ and $A/yA\cong \Lambda_\mu(V)$. By Proposition~\ref{PBW}, 
$\gr(R) \cong \Lambda_\mu(V)$, and so, by \cite[Theorem~1.3]{cassidy&shelton}, 
$y$~is a regular element in~$A$. 
 It is also clear that if the matrix $B$~is not identically zero, then $y$~is a 
redundant generator, so that $A$~is in fact generated by degree-one elements.  In 
this case, it turns out that  the centrality of~$y$ is a consequence of the 
degree-two relations, as the next theorem shows.

\begin{theorem}  If $R$ has full dimension and $B \neq 0$, then $A$~is a 
quadratic algebra.
\end{theorem}
 
\begin{proof}  We show that the quadratic relations in~$A$ imply that $y$ is central.
Firstly, suppose $B_{ii} \neq 0$ for some~$i$.  Without loss of
generality, take $i = 1$, so $y=X_1^2$ in~$A$. Since $R$~has full dimension, 
condition~$\bigstar$ implies that $\mu_{1k}^2=1$ for all~$k$.  It follows that 
$X_k y=y X_k+ 2 B_{k1}(1-\mu_{k1}) X_1^3$ for all~$k$.  However, 
condition~$\bigstar$ implies that  $B_{k1}(1-\mu_{k1})=0$, and so   
$X_ky=yX_k$ for all $k$.
 
Now suppose $B_{ii} = 0$ for all $i$.  Since $B$~is nonzero,
we may assume, without loss of generality, that $2B_{12}=1$, so that $\mu_{12}=1$
(by $\bigstar$) and $y=X_1X_2+X_2X_1$ in $A$.     If $B_{k1}\ne 0$ for some $k\ne
1$, then $2B_{k1}y=X_1X_k+X_kX_1$, so that
$X_ky=(2B_{k1})^{-1}X_k(X_1X_k+X_kX_1)=(2B_{k1})^{-1}X_k
X_1X_k=(2B_{k1})^{-1}(X_1X_k+X_kX_1)X_k=yX_k$.  Similarly, if $B_{k2}\ne 0$ for
some $k\ne 2$, then $X_k$ commutes with~$y$.  If   $B_{k1}= 0=B_{k2}$ for some
$k\ge 3$, then $X_ky=\mu_{k1}\mu_{k2}yX_k$, and since $B_{12}\ne 0$, 
condition~$\bigstar$ implies that $\mu_{k1}\mu_{k2}=1$.
\end{proof}
    
To relate skew Clifford algebras of full dimension to the graded skew Clifford
algebras studied in~\cite{set}, we construct a nonstandard homogenization of~$R$
by adjoining $n$ degree-two central elements $y_1,\ldots , y_n$.  This new algebra 
will depend on a choice of $\mu$-symmetric matrices $M_1, \ldots , M_n \in 
M(n, \, \kk)$ that satisfy $\sum_{k=1}^n(M_k)_{ij} =2B_{ij}$ for all $i, j$.  
We  define the $\N$-graded $\kk$-algebra  $A(n)$ as
\begin{equation}\label{nonstand}
A(n):=\frac{\kk\la X_1,\ldots , X_n\ra[y_1,\ldots , y_n]}{\la X_iX_j+\mu_{ij}X_jX_i-\sum_{k=1}^n(M_k)_{ij}y_k : 1\le i, j\le n\ra},
\end{equation}
 i.e. $A(n)$ has  $n(2n-1)$ relations that make $y_k$ central in $A(n)$ for all $k$, 
and another $({n^2+n})/2 $ relations derived from the   relations in $R$.   
 The algebra $R$ is a ``dehomogenization'' of $A(n)$  in that $A(n)/\la
 y_1-1,\ldots , y_n-1\ra \cong R$.  
The algebra $A(n)$ is  an {\it iterated central extension} of $\Lambda_\mu(V)$ in
that each $y_k$ is central in~$A(n)$ and 
$A(n)/\la y_1,\ldots ,  y_n\ra\cong \Lambda_\mu(V)$. 
\begin{figure}[h] 
 \center{
\begin{tikzpicture}
\matrix (m) [matrix of math nodes, row sep=1em,
column sep=2em, text height=1.5ex, text depth=0.25ex]
{& &     R \  \   \\
A(n)& A &  \\ 
&  & \   \ \Lambda_\mu(V)   \\
};
\path[->>]  (m-2-1) edge     (m-2-2) ;
 \path[->>]  (m-2-2) edge     (m-1-3);
\path[->>] (m-2-2) edge    (m-3-3);
\end{tikzpicture}}
\end{figure}
   Note that, by construction, $X_j^2$ is central in $A(n)$ for all $j$. Moreover, 
$A(n)$ maps onto~$A$, and, as is the case with~$A$, it is possible for $A(n)$~to 
be a quadratic algebra generated by degree-one elements; Theorems~\ref{quotient} 
and \ref{ASreg} address this possibility.
   
The algebra $A(n)$ is a graded skew Clifford algebra, as defined in~\cite{set}.

\begin{definition} \label{GSCA} \cite{set}
 Fix $\mu$-symmetric matrices $M_1, \ldots , 
M_n \in M(n, \, \kk)$. 
A {\em graded skew Clifford algebra} 
associated to $\mu$ and $M_1$, $\ldots ,$ $M_n$ is a graded $\kk$-algebra
on degree-one generators $x_1, \ldots , x_n$ and on degree-two generators
$y_1, \ldots , y_n$, with defining relations as follows:
\begin{enumerate}
\item[(a)] degree-two relations of the form $ x_i x_j + \mu_{ij} x_j x_i = \sum_{k=1}^n (M_k)_{ij} y_k$
           for all $i, j \in\{1, \ldots , n\}$, and
\item[(b)] degree-three and degree-four relations (possibly obtainable from the quadratic relations in (a)) that guarantee the existence of a normalizing sequence $\{ r_1, \ldots , r_n\}$
           of homogeneous elements of degree two that span 
	   $\kk y_1 + \cdots + \kk y_n$. 
\end{enumerate}
\end{definition} 

The notion of graded skew Clifford algebra is a generalization of the notion of 
graded Clifford algebra (cf.\ \cite{set,lebruyn}).  

We now show that if $\phi\ne 0$, then the skew Clifford algebra $R$ is the image
of a graded skew Clifford algebra that is generated by degree-one elements.
Consider a skew Clifford algebra $\sCl(V, \, \mu, \, \phi)$, where $V$~has
basis $\mathcal B = \{x_1,\ldots , x_n\}$ and $\mu$-symmetric bilinear 
form~$\phi$ with respect to~$\mathcal B$, and let 
$D = (D_{ij})$~be the $\mu$-symmetric matrix given by $D_{ij}=2\phi(x_i,\, x_j)$.

\begin{lemma}\label{partition}
Suppose $\mu \in M(n, \, \kk)$ is multiplicatively antisymmetric  and 
$D\in M(n, \, \kk)$ is $\mu$-symmetric.  If $D \neq 0$, then 
$D=\sum_{i=1}^n M_i$, where $M_1, \ldots , M_n$ are linearly independent 
$\mu$-symmetric matrices.
\end{lemma}

\begin{proof}  We proceed by induction on $n$.  If $n=1$, the result is obvious.  
Let $\mu' \in M(n-1, \, \kk)$ be the multiplicatively antisymmetric matrix formed 
from the first $n-1$ rows, and first $n-1$ columns, of~$\mu$.   Assume inductively 
that every $\mu'$-symmetric matrix in $M(n-1, \, \kk)$ can be written as a sum of 
$n-1$~linearly independent $\mu'$-symmetric matrices.  Let $D' \in M(n-1, \, \kk)$ 
be the $(n-1)\times (n-1)$ upper left block of~$D$, and notice that  $D'$~is 
a $\mu'$-symmetric matrix.  We consider two cases.

I. Case $D' \neq 0$.  By the induction hypothesis, there exist $n-1$~linearly 
independent $\mu'$-symmetric matrices $M'_i \in M(n-1, \, \kk)$ such that 
$D'=\sum_{i=1}^{n-1}M'_i$.  If $D_{in}\ne 0$, for some~$i$, then
we define $M_1,\ldots , M_n \in M(n, \, \kk)$ as follows.    Let 
$$M_k=\begin{pmatrix}
      \multicolumn{2}{c}{\kern0.5em\smash{\raisebox{-2ex}{\large $M'_k$}}} &0\\
      & &\vdots  \\
     0& \cdots & 0
    \end{pmatrix} \  \text{for $k<n$, \quad and let}\ \  M_n=\begin{pmatrix}
      \multicolumn{2}{c}{\kern0.5em\smash{\raisebox{-2ex}{\huge 0}}} &D_{1n}\\
      & &\vdots  \\
    \ D_{n1} & \cdots & D_{nn}
    \end{pmatrix}.$$
    These matrices are $\mu$-symmetric, and the linear independence of 
    $M'_1, \ldots, M'_{n-1}$ ensures that $M_1, \ldots, M_n$ are also linearly 
    independent. 
On the other hand, if $D_{in}=0=D_{ni}$ for all $i$, then we define 
$M_2,\ldots , M_{n-1}$ as before, but define 
$$M_1=\begin{pmatrix}
      \multicolumn{2}{c}{\kern0.5em\smash{\raisebox{-2ex}{\large $M'_1$}}} &0\\
      & &\vdots  \\
     0& \cdots & 1
    \end{pmatrix} \quad \text{and} \quad M_n=\begin{pmatrix}
      \multicolumn{2}{c}{\kern0.5em\smash{\raisebox{-4ex}{\LARGE 0}}} &0\\
      & &\vdots  \\
            & &0  \\
     0 & \cdots \ 0 & -1
    \end{pmatrix}.$$
    
  II.  Case $D' = 0$. For $k<n$, let $M_k$ be the matrix with  $1$ in the
  $(k,k)$-entry and zero elsewhere, and define $M_n=D-\sum_{k=1}^{n-1}M_k$.
  Since $D$ is nonzero, $D_{in}\ne 0$ for some $i$.  
  It follows that    $M_1,\ldots ,M_n$ are linearly independent.
\end{proof}
Note that the partition of~$D$ given in the proof of Lemma~\ref{partition} is not 
necessarily the only way to decompose~$D$ into $\sum_{i=1}^n M_i$, where
$M_1,\ldots ,M_n$ are linearly independent $\mu$-symmetric matrices.

\begin{theorem}\label{quotient} 
With notation as in Definition~\ref{scliff2}, if $\phi \neq 0$, then 
$\sCl(V,\, \mu, \, \phi)$ is  a quotient of a graded skew Clifford algebra 
that is generated by degree-one elements.
\end{theorem}
\begin{proof}
Let $\mathcal B = \{x_1,\ldots , x_n\}$ be a fixed ordered basis of~$V$, and let 
$B$~be the matrix of~$\phi$ with respect to $\mathcal B$. As a consequence of  
Lemma~\ref{partition}, we can
choose a partition of~$2B$ into linearly independent $\mu$-symmetric matrices
$M_1, \ldots , M_n$. We use these matrices, $M_1, \ldots , M_n$, to define a 
graded skew Clifford algebra $A(n)$ as in  (\ref{nonstand}).  By construction, 
$A(n)/\la y_1-1,\ldots , y_n-1\ra \cong \sCl(V,\, \mu, \, \phi)$. 
Since $M_1, \ldots , M_n$ are linearly independent, \cite[Lemma~1.13]{set} implies 
that  $A(n)$~is generated by   $X_1,\ldots , X_n$.     
\end{proof}

Thus, every skew Clifford algebra $R = \sCl(V,\, \mu, \, \phi)$, with 
nonzero~$\phi$, can be homogenized (in the above manner) to create a  graded 
skew Clifford algebra that is generated by degree-one elements.  The converse is 
true for (classical) graded Clifford algebras: every
graded Clifford algebra is a nonstandard homogenization of a classical Clifford
algebra; but this converse does not hold in general for their skew analogues (as
defined in this article). In particular, the graded algebra produced by this  
homogenization contains central elements of degree two given by the~$y_k$.  Since 
graded skew Clifford algebras need not have central elements of degree two, not  
every graded skew Clifford algebra occurs as a  homogenization of a skew Clifford 
algebra. 

The homogenized algebra $A(n)$ maps onto $R$  and is generated by the pre-images of the 
generators of~$R$.  Although the linear independence of the matrices $M_1, \ldots
, M_n$ guarantees, via Theorem~\ref{quotient}, that $A(n)$~can be generated by 
degree-one elements, $A(n)$~might 
not  be quadratic.   Nevertheless, $A(n)$ is noetherian, and although it is 
noncommutative, it can have some other nice properties of the commutative polynomial 
ring on $n$~variables, such as being Auslander regular~(\cite{Bjork}), 
Artin-Schelter regular~(\cite{ArtinSchelter}), or a domain.  As discussed 
in~\cite{set}, if the graded skew Clifford algebra~$A(n)$ is Auslander regular,
then it is also Artin-Schelter regular and a domain.  Since the 
family~$\mathcal F$ of graded skew Clifford algebras includes many 
Artin-Schelter regular algebras (cf.\ \cite{NVZ}), it is natural to ask which 
skew Clifford algebras are quotients of quadratic  regular algebras in~$\mathcal F$;  
this issue is addressed in our final result.

\begin{theorem}\label{ASreg}
With notation as in Definition~\ref{scliff2}, let $R = \sCl(V,\, \mu, \, \phi)$,
where $\phi \neq 0$ and $\dim(V) = n < \infty$.  If $R$ has full dimension, 
then we can choose $\mu$-symmetric 
matrices $M_1,\ldots , M_n \in M(n, \, \kk)$ such that the graded skew Clifford 
algebra~$A(n)$  associated to $\mu$ and $M_1,\ldots , M_n$ is a quadratic 
Auslander-regular algebra if and only if $\mu_{ij}^2=1$ for all $i, j$. 
\end{theorem}
\begin{proof}
Suppose $R$ has full dimension.
As before, we write $B_{ij} = \phi(x_i, \, x_j)$.
Suppose first that $A(n)$ is quadratic and Auslander regular.    If $B_{ij}\ne 0$, 
then, by condition~$\bigstar$, we have $\mu_{ik}=\mu_{kj}$ for every $k$. 
In particular, by taking $k=j$, we obtain $\mu_{ij}=\mu_{jj}=1$.  On the other hand, 
if $B_{ij}=0$, consider $X_iX_j+\mu_{ij}X_jX_i=\sum_k(M_k)_{ij}y_k$.  For 
convenience, let $W$  denote $\sum_k(M_k)_{ij}y_k$, and note that the image of~$W$ 
in~$R$ is zero.  Since $W$ and $X_j^2$ are central in $A(n)$, we have $X_j^2X_i=X_iX_j^2=\mu_{ij}^2X_j^2X_i+X_j(1-\mu_{ij})W$.  Since $A(n)$ is  a domain, $X_jX_i=\mu_{ij}^2X_jX_i+(1-\mu_{ij})W$, and thus $(1-\mu_{ij}^2)X_jX_i=(1-\mu_{ij})W$. 
Projecting this equation onto~$R$ gives $(1-\mu_{ij}^2)x_jx_i=0$, but, since 
$R$~has full dimension, $x_jx_i \neq 0$ in $R$, so $\mu_{ij}^2= 1$.
 
For the converse, suppose that $\mu_{ij}^2=1$ for all $i, j$.   If $\dim(V) =1$,
then $R\cong \kk[x]/\la x^2-1\ra$ and 
$\Lambda_{\mu}(V)\cong \kk[x]/\la x^2\ra$. So, in this case, $A(n)=\kk[x]$ and 
is an Auslander-regular graded skew Clifford algebra.  From  now on, assume that 
dim$(V)>1$. We will construct a homogenization~$A(n)$ (as defined earlier in this
section) of~$R$ that is Auslander regular and generated by elements of degree one.   
 
Consider first the case that $B$~is diagonal, so $B_{ij}=0$ for all $i\ne j$.  
Since $B\ne 0$, we may assume, without loss of generality, that  $B_{11}\ne 0$.   
We define diagonal matrices $M_1, \ldots , M_n$ by 
\[(M_k)_{ij}= \left\{
\begin{array}{ll}
2B_{kk}  &   \text{ if } k=i=j \text{ and } B_{kk} \neq 0,\\[2mm] 
  -1 &   \text{ if }   1 \neq k = i = j \text{ and }   B_{kk}=0,\\[2mm] 
  1 &   \text{ if } 1 = k  < i = j \text{ and } B_{jj}=0,\\[2mm] 
  0  & \ \text{otherwise.}
    \end{array}  \right. 
\] 
It follows that $\sum_{k=1}^n(M_k)_{ij} =2B_{ij}$ and that $M_1, \ldots, M_n$
are linearly independent, so that the $y_k$ are redundant generators for $A(n)$.  
Moreover, since  $\mu_{ij}^2=1$ and each $M_k$~is diagonal,  the quadratic relations 
imply that $X_iX_j^2=X_j^2X_i$ for all $i, j$, giving that $X_j^2$  is central for
all $j$.  Furthermore, in $A(n)$, we have
   $$y_j=\left\{\begin{array}{l l}
 X_j^2/B_{jj}  &   \text{ if }   B_{jj}\ne 0, \\[3mm]
(X_{1}^2/B_{11}) - 2X_j^2 &   \text{ if }   B_{jj}=0,
    \end{array}  \right. $$ 
so $y_j$ is central for all $j$.  It follows that, in this case, $A(n)$~is the 
skew polynomial ring 
\[
\kk\la X_1,\ldots , X_n\ra/\la X_iX_j+\mu_{ij}X_jX_i : 1 \leq i < j \leq n\ra,
\] 
which is a quadratic Auslander-regular algebra. 
   
Now suppose that $B$ is not diagonal.   Since $B\ne 0$, by reordering the~$x_i$ 
and scaling, we may assume $B_{12}=1$. Since $R$~has full dimension, 
condition~$\bigstar$ implies that $\mu_{12}=1$.    
Similarly, if $B_{11}B_{22}=0$, we may assume $B_{22}=0$. We define 
$\mu$-symmetric matrices $M_1, \ldots, M_n$ as follows:
\[
(M_1)_{ij} = \left\{ 
\footnotesize
           \begin{array}{ll}
	   1 & \text{ if } i = j = 1 \text{ and } B_{11} = 0,\\[2mm]
	   -1 & \text{ if } i = j \geq 2 \text{ and } B_{jj} = 0,\\[2mm]
	   0 & \text{ if } i = j \geq 2 \text{ and } B_{jj} \neq 0,\\[2mm]
	   2B_{ij} & \text{ otherwise,} 
           \end{array} \right.
\quad
\normalsize
(M_k)_{ij} = \left\{ 
\footnotesize
           \begin{array}{ll}
	   1 & \text{ if } i = j = k \text{ and } B_{kk} = 0,\\[2mm]
	   2B_{kk} & \text{ if } i = j = k \text{ and } B_{kk} \neq 0,\\[2mm]
	   -1 & \text{ if } k=2, \ i = j = 1 \text{ and } B_{11} = 0,\\[2mm]
	   0 & \text{ otherwise,} 
           \end{array} \right. 
\]
where $k \geq 2$.

The matrices $M_k$ satisfy $\sum_{k=1}^n(M_k)_{ij} =2B_{ij}$ and are
$\mu$-symmetric and linearly independent since $2B_{12} \neq 0$. Moreover,  
$A(n)/\la y_1-1,\ldots , y_n-1\ra \cong R$ and 
$A(n)/\la y_1,\ldots ,  y_n\ra\cong \Lambda_\mu(V)$.  We now show that  the 
centrality of the $y_i$ is a consequence of the quadratic relations, so that 
$A(n)$~is in fact a quadratic algebra.  

We first consider~$y_1$. If $B_{11}=0=B_{22}$, then $2X_1^2=y_1-y_2$, 
$X_1X_2+X_2X_1=2y_1$ and $2X_2^2=-y_1+y_2$.  In particular, $X_1^2=-X_2^2$, which 
implies that 
$y_1$~commutes with $X_1$ and $X_2$.  However, if $B_{11}B_{22}\ne 0$, then 
$X_1^2=B_{11}y_1$, $X_1X_2+X_2X_1=2y_1$ and $X_2^2=B_{22}y_2$. In this case, we 
have $y_1X_1=X_1y_1$ and
\[
B_{11} y_1 X_2 = X_1^2 X_2 = 2 X_1 y_1 - X_1 X_2 X_1 = X_2 X_1^2 = B_{11} X_2 y_1
,\]
so, again,  $y_1$~commutes with $X_1$ and $X_2$.  
On the other hand, if $B_{11}B_{22}=0$ and $(B_{11}, \, B_{22})\ne(0, \, 0)$, we 
may assume, as before, that $B_{22}=0\ne B_{11}$.  In this case, 
$X_1^2=B_{11}y_1$, $X_1X_2+X_2X_1=2y_1$ and $2X_2^2=y_2-y_1$.  
Thus $y_1X_1=X_1y_1$.  Moreover, resolving $X_1^3$ implies that 
$X_1^2X_2=X_2X_1^2$, from which we find $y_1X_2=X_2y_1$. Hence, in all three 
scenarios, $y_1$ commutes with $X_k$ for $k=1, 2$.

For $k>2$, we use the quadratic relations $X_kX_i+\mu_{ki}X_iX_k=2B_{ki}y_1$,
$i=1, 2$, together with $2 y_1 = X_1 X_2 + X_2 X_1$,  to calculate 
$$X_ky_1=
\mu_{k1}\mu_{k2}y_1X_k+ B_{k1} (1-\mu_{k2})y_1X_2+ B_{k2} (1-\mu_{k1})y_1X_1.$$  
As $B_{12}\ne 0$, condition $\bigstar$ implies that $\mu_{k1}\mu_{k2}=1$. 
If $B_{k1}\ne 0$, then condition $\bigstar$ implies that $\mu_{k2}=\mu_{21}=1$,
and similarly if $B_{k2}\ne 0$ then $\mu_{k1}=1$.  Thus, $y_1$ commutes with 
$X_k$ for all $k$.  
  
Now observe that, for $j>1$, we have 
  $$y_j=\left\{\begin{array}{l l}
 X_j^2/B_{jj}  &   \text{ if }   B_{jj}\ne 0,  \\[3mm]
2X_j^2+y_1 &   \text{ if }   B_{jj}=0. 
    \end{array}  \right. $$ 
Thus, for $j > 1$, $y_j$ commutes with $X_k$, for all~$k$, provided that 
$X_j^2$ commutes with $X_k$, for all~$k$.  One may verify this is indeed the case 
by subtracting $(X_k X_j + \mu_{kj} X_j X_k = 2 B_{jk} y_1)X_j$ from 
$X_j (X_j X_k + \mu_{jk} X_k X_j = 2 B_{kj} y_1)$, and applying 
condition~$\bigstar$ and our hypothesis that $\mu_{jk}^2=1$ for all $j, k$.
Lastly, since each $y_j$ commutes with each $X_k$ and since 
$y_j$ is a linear combination of products of the form $X_aX_b$, the $y_j$ 
commute with each other. It follows that the centrality of the~$y_k$ in~$A(n)$ 
is a consequence of only the degree-two relations of~$A(n)$,  and so $A(n)$
is quadratic.

We now show that $A(n)$ is an Auslander-regular algebra by examining 
two cases. Without loss of generality,  we may assume that either 
$B_{11}=B_{22}=0$ or, by reversing the roles of $x_1$ and $x_2$ if necessary, 
that $B_{11}\ne 0$. 

Case I: $B_{11}=B_{22}=0$.  
Since $y_k \in (A(n)_1)^2$, for all~$k$, and since $A(n)$ is quadratic,   
we can present $A(n)$ as 
$$\frac{\kk\la X_1,\ldots , X_n\ra }{\la X_1^2+X_2^2, \
B_{ij}(X_1X_2+X_2X_1)-(X_iX_j+\mu_{ij}X_jX_i) : 
1 \leq i \neq j \leq n \ra}.$$  
This algebra is an iterated Ore extension of  $ {\kk\la X_1,X_2 \ra}/{\la
X_1^2+X_2^2 \ra}$. For $k\geq 3$, the automorphisms~$\sigma_k$ used to adjoin
$X_k$  are given by  $\sigma_k(X_i)=-\mu_{ki}X_i$,  and the corresponding left
$\sigma_k$-derivations are  given by $\delta_k(X_i)= B_{ki} (X_1X_2+X_2X_1)$.  To
see that $\sigma_k$ is well defined, we use the hypothesis that
$\mu_{k1}^2=1=\mu_{k2}^2$ to obtain $\sigma_k(X_1^2+X_2^2)=0$, and we use  condition $\bigstar$ to establish that $\sigma_k(B_{ij}(X_1X_2+X_2X_1)-(X_iX_j+\mu_{ij}X_jX_i))=0$. 
Turning to~$\delta_k$, we calculate that, for any $i\ne j$, we have
$$\delta_k(X_iX_j+\mu_{ij}X_jX_i)=(B_{ki}(1-\mu_{kj}\mu_{ij})X_j+B_{kj}(\mu_{ij}-\mu_{ki})X_i)(X_1X_2+X_2X_1).$$
Condition $\bigstar$ implies that if $B_{ki}\ne 0$, then 
$\mu_{kj}=\mu_{ji}=\mu_{ij}^{-1}$, and if $B_{kj}\ne 0$, then $\mu_{ki}=\mu_{ij}$.  
Thus $\delta_k(X_iX_j+\mu_{ij}X_jX_i)=0$ for all $i\ne j$. Similarly, $\delta_k(X_1^2+X_2^2)=0$. Hence $\delta_k$ is well defined.   Since $ {\kk\la X_1,X_2 \ra}/{\la X_1^2+X_2^2 \ra}$ is  Auslander regular,  
and, by~\cite{Lev&Staf}, Ore extensions preserve this property, $A(n)$ is also 
Auslander regular.

Case II: $B_{11}\ne 0$.  By rescaling the $x_i$  if necessary, we may assume that
$B_{11} = 2$ while maintaining $B_{12} = 1$.
Since $y_k \in (A(n)_1)^2$, for all~$k$, and since $A(n)$ is quadratic,   
we can present $A(n)$ as 
$$\frac{\kk\la X_1,\ldots , X_n\ra }
{\la X_1X_2+X_2X_1-X_1^2, \ B_{ij}(X_1X_2+X_2X_1)-(X_iX_j+\mu_{ij}X_jX_i) :
1 \leq i \neq j \leq n \ra}.$$
Using the above formulae for the maps~$\sigma_k$ and~$\delta_k$, this 
algebra is an iterated Ore extension of 
$$\frac{ \kk\la X_1,X_2 \ra}{\la X_1X_2+X_2X_1-X_1^2 \ra},$$ 
which is  Auslander regular by~\cite{ArtinSchelter}.   In particular, 
an argument similar to that in Case~I shows that the $\sigma_k$ and~$\delta_k$ 
are well defined (but, in computing $\delta_k (X_1X_2+X_2X_1-X_1^2)$,  one 
should note that condition~$\bigstar$ implies that if $B_{ki} \neq 0$, where 
$i = 1, 2$, then $\mu_{kj} = \mu_{ji}$, which equals $1$ if $j \in \{1, \, 2\}$).
Again, as in Case~I, it follows from~\cite{Lev&Staf} that $A(n)$~is 
Auslander regular.

We conclude that, in all cases, $A(n)$~is a quadratic, noetherian, Auslander-regular 
domain.
\end{proof}

Theorem~\ref{ASreg} and Remark~\ref{diag} imply that if 
$B$~is an invertible diagonal matrix, then any corresponding skew 
Clifford algebra of full dimension is a quotient of a quadratic 
Auslander-regular algebra.

\bigskip


\section{Other Generalizations}\label{generalizations}

In this section we compare skew Clifford algebras to other generalizations of classical Clifford algebras found in the literature.

There are  at least three different constructions called ``generalized Clifford
algebras,'' all of which extend the classical notion by allowing 
defining relations of degree greater than two.  N.~Roby's generalization of
Clifford algebras in~\cite{Roby} was subsequently extended  by C.~Pappacena 
in~\cite{Pappacena} in his study of matrix pencils.  Pappacena shows that when the
defining relations are degree three or higher, these generalized Clifford algebras have infinite dimension.  He also proves that when the  defining relations are quadratic, 
and char$(\kk) \ne 2$, then these algebras are isomorphic to classical 
Clifford algebras. 

C.~Ko\c{c} in~\cite{Koc} approaches the problem of generalizing Clifford algebras
from a different direction.  Given a vector space $V$ with quadratic form $Q$,
Ko\c{c}  looks for algebras $A$  that share the following property of classical Clifford algebras.
\begin{Condition} 
\label{koc}
Every isometry of $V$ with respect to the quadratic form $Q$ can be extended to an automorphism of $A$, and conversely, every automorphism of $A$ mapping $V$ to itself induces an isometry on $V$ with respect to $Q$.
\end{Condition}  
Ko\c{c}  describes a family of algebras that satisfy this condition and it
contains, as the quadratic members, the classical Clifford algebras; the 
other algebras in Ko\c{c}'s family are not quadratic.     
 
Morris, in~\cite{Morris}, defines generalized Clifford algebras that are presented
with generators $x_1, \ldots , x_n$ and relations $x_ix_j=\omega x_jx_i$  and $x_i^N=1$,
where $\omega$ is a primitive $N^{th}$ root of unity. These generalized Clifford
algebras, like those of Pappacena and Ko\c{c}, coincide with our skew Clifford
algebras only if the algebra is a classical Clifford algebra.   

Skew Clifford algebras also bear some resemblance to the quantum Clifford algebras
given in~\cite{frenkel&ding} (see also \cite{Benkart&P}).  Quantum Clifford 
algebras are presented with $2n$~generators,
$x_1,\ldots , x_n,x_1^*,\ldots, x_n^*$, with defining relations
$$
x_ix_j=-q^{-1}x_jx_i, \hskip1cm  x^*_ix^*_j=-qx_j^*x_i^* \ \ \  (i>j),  $$
$$ x_ix_i=0=x^*_ix^*_i, \hskip1cm x_ix^*_j=-q^{-1}x^*_jx_i \ \ \ (i\ne j),$$
 $$x_ix_i^*+x^*_ix_i=(q^{-2}-1)\sum_{i<j}x_jx_j^*+1.$$
\quad\\[-1mm]
A quantum Clifford algebra is a skew Clifford algebra  if $q=\pm 1$. The case $q=1$ yields a classical Clifford algebra. 
 
Closer in spirit to our skew Clifford algebras is the recent work of Chen and Kang 
in~\cite{CK}, which develops a generalized Clifford theory on group-graded vector 
spaces using bicharacters.  Their construction is quite general, and includes 
classical Clifford algebras, Weyl algebras, and polynomial rings.   

\bigskip
\bigskip
\bigskip

\defbibheading{bibliography}{\centering \sc References\\[6mm]}

\renewcommand{\normalsize}{\footnotesize}
\setlength\bibitemsep{0.6\baselineskip}

\printbibliography

\bigskip
\bigskip

\end{document}